\documentclass[ECP,preprint]{ejpecp} 

\usepackage{filecontents} 

\usepackage{bbm}

\newcounter{rownum}

\SHORTTITLE{Representation-free Densities of Stable Ratios}

\TITLE{Representation-free Densities of Ratios and Powers of Stable Variables} 

\AUTHORS{%
  Nomvelo ~Karabo~Sibisi\footnote{University of Cape Town.
    \EMAIL{sbsnom005@myuct.ac.za}}
  \and 
  Sibusiso~Sibisi\footnote{\EMAIL{ssibisi@gmail.com}}}

\AMSSUBJ{60E05; 60G52} 

\SUBMITTED{September 15, 2026} 
\ACCEPTED{xxx} 

\VOLUME{0}
\YEAR{2026}
\PAPERNUM{0}
\DOI{???}

\ABSTRACT{
We derive representation-free  densities of  ratios  of the positive  stable random variable
and obtain their Laplace/Stieltjes  transforms. 
We  state representation-free densities of occupation times (Lamperti~\cite{Lamperti})
and  Bessel process excursions  (Bertoin~{\it et al.}~\cite{Bertoin}) and their  transforms.
We also  obtain   simple representations of  these and other densities using a known infinite series representation 
of the stable density.
}

\newcommand\Tstrut{\rule{0pt}{3.5ex}}       
\newcommand\Bstrut{\rule[-3.25ex]{0pt}{0pt}} 
\newcommand{\TBstrut}{\Tstrut\Bstrut} 

\begin{document}

\section{Introduction}
\label{sec:introduction}

The probability density  of the positive  stable random variable  $S_\alpha$ $(0<\alpha<1)$ 
does not have a simple representation  (with the notable exception of $\alpha=1/2$).
It does have  a known  infinite series representation but, in the first instance, 
we shall treat the stable density as  ``representation-free'' , indirectly represented by  its  simple Laplace transform.
Yet,  as we  demonstrate, such minimal knowledge suffices to express the  densities of a variety of 
products and ratios of stable variables in equally representation-free form and  to evaluate  explicit forms of their  
forward Laplace and Stieltjes transforms.
The approach may thus be described as both representation-free and inversion-free.

We then show that 
substituting the infinite series representation 
of the stable density  can yield very simple expressions for the densities.
Many appear in the probability  literature, although approaches to derive them vary.
A classic example  is  the 
derivation  by Chaumont and Yor~{\rm \cite[4.23.3]{ChaumontYor}} 
of the  density of  the ratio  $(S_\alpha/S^\prime_\alpha)^{\alpha}$ 
for  $S^\prime_\alpha \overset{d}{=} S_\alpha$.

Lamperti~\cite{Lamperti} derived a closely related density  by  taking  the Stieltjes inverse of 
a certain function  (further detail in Section~\ref{sec:lamperti}).  
Although he did not  refer to the stable density,  it is known that the Lamperti density arises from a  stable ratio
(James~{\rm \cite{James_Lamperti}). 
We  start from the appropriate ratio to derive a representation-free form of the 
Lamperti density and hence the forward Stieltjes transform matching  Lamperti's starting point.
We then complete the picture by substituting the infinite series representation of the stable density 
to obtain the simple form that Lamperti derived by Stieltjes inversion.

In a study ``on a particular class of self-decomposable random variables", 
Bertoin~{\it et al.}~\cite{Bertoin} (BFRY) derived a  function bearing a resemblance 
to the one  due to  Lamperti, whose  Stieltjes inverse is  a similarly simple density.
 BFRY also identified the  corresponding stable variable ratio. 
 However, derivation of  the  forward Stieltjes transform  from the density proved rather challenging.  
 While representation-free densities are not available for direct computation,  
 they allow easier derivation of forward  Stieltjes transforms. 
 
 We also  discuss ``bottom-up'' construction of  
 generalised gamma convolutions (GGCs),
 a  class of 
 densities described by Bondesson~\cite{Bondesson}.
 This contrasts with a ``top-down" construction which calls for  concepts such as Dirichlet means to establish existence, 
 as discussed in James~{\it et al.}~\cite{JamesRoynetteYor} (JRY), where the BFRY example features prominently.
 


\section{Preliminaries}
\label{sec:prelim}

\begin{definition}[Stable Distribution]
\label{def:stable}
The distribution of the positive stable   variable $S_{\alpha;z}$ for  $(0<\alpha<1)$ and scale factor $z>0$,
 is indirectly defined by  the Laplace  transform 
\begin{align}
 \mathbb{E}\left[e^{-x S_{\alpha;z}}\right]  &\equiv  \int_0^\infty e^{-x t}  f_\alpha(t\vert z)\, dt = e^{-z x^\alpha} \quad (x\ge0)
 \label{eq:stableLT}
\end{align}
where $\Pr(S_{\alpha;z}\equiv t) = f_\alpha(t\vert  z)$ $(t>0)$ is the stable density.  
\begin{remark}
\label{rem:stableProperty}
From $(\ref{eq:stableLT})$,   $f_\alpha(t\vert z) \equiv f_\alpha(t z^{-1/\alpha}\vert 1) z^{-1/\alpha}
\implies tf_\alpha(t\vert z) \equiv f_\alpha(1\vert z t^{-\alpha})$.
\end{remark}
 \end{definition}

\begin{proposition}
\label{prop:stablePower}
 Let $S^{-\beta}_{\alpha;z}$  be a stable variable power for $\beta>0$ and  $M_{\alpha;z}\equiv S^{-\alpha}_{\alpha;z}$. 
 Then
\begin{align}
  \mathbb{E}\left[S^{-\beta}_{\alpha;z} \right]  &=  z^{-\beta/\alpha} \frac{\Gamma(\beta/\alpha+1)}{\Gamma(\beta+1)}
 \label{eq:stablePowerMoment}  \\
 \mathbb{E}\left[M^k_{\alpha;z} \right]  
   &=  z^{-k} \frac{k!}{\Gamma(\alpha k+1)} \quad  (k\ge0)  
 \label{eq:MLMoment}  \\
 \mathbb{E}\left[e^{-x M_{\alpha;z}}\right]  
    &=    \sum_{k=0}^\infty \frac{(-x/z)^k}{\Gamma(\alpha k+1)} = E_\alpha(-x/z) \quad (x\ge0)
 \label{eq:MLLT}
\end{align}
$E_\alpha()$ is the Mittag-Leffler function and
$M_{\alpha;z} \equiv S^{-\alpha}_{\alpha;z}$   is the Mittag-Leffler variable. 
\end{proposition}
\begin{proof}[Proof of Proposition~$\ref{prop:stablePower}$] 
\label{proof:stablePower}
Janson \cite[Example~3.10]{JansonProbSurv} proved $(\ref{eq:stablePowerMoment})$ $(z=1)$ 
(also see Shanbhag and Sreehari \cite[Theorem~1]{Shanbhag}), restated here for $z>0$.
With  $f_\alpha(t \vert z)$ unspecified, we cannot  evaluate
$\mathbb{E}\bigl[S^{-\beta}_{\alpha;z} \bigr]  \equiv  \int_0^\infty t^{-\beta} f_\alpha(t \vert z) dt$.
But $\Gamma(\beta) t^{-\beta} = \int_0^\infty e^{-tu} u^{\beta-1}du$, hence 
\begin{align*}
\Gamma(\beta)\, \mathbb{E}\left[S^{-\beta}_{\alpha;z} \right]  
   &= \int_0^\infty u^{\beta-1}  \int_0^\infty  e^{-tu} \, f_\alpha(t \vert z) \, dt \, du \\
   &= \int_0^\infty u^{\beta-1} e^{-zu^\alpha} \,  du = \frac{1}{\alpha} \int_0^\infty u^{\beta/\alpha-1} e^{-zu} \,  du \\
   &= z^{-\beta/\alpha} \frac{1}{\alpha} \, \Gamma(\beta/\alpha) 
        = z^{-\beta/\alpha} \frac{1}{\beta} \, \Gamma(\beta/\alpha+1) 
\end{align*}
which gives (\ref{eq:stablePowerMoment}). 
Moments~(\ref{eq:MLMoment})  and  Laplace transform~(\ref{eq:MLLT})  follow readily.
\end{proof}
\begin{remark}
\label{rem:janson}
Proof of Proposition~$\ref{prop:stablePower}$ replaces the intractable integral 
$\mathbb{E}\left[S^{-\beta}_{\alpha;z} \right]$ by a  double integral and  swaps the  integration order (by Fubini-Tonelli) 
to obtain $\mathbb{E}\left[S^{-\beta}_{\alpha;z} \right]$ based  on  
$\mathbb{E}\left[e^{-x S_{\alpha;z}}\right] = \exp(-z x^\alpha)$. 
$\beta=\alpha$ gives the moments $\mathbb{E}\left[M^k_{\alpha;z} \right]$, 
leading to the Laplace transform $\mathbb{E}\left[e^{-x M_{\alpha;z}}\right]$ without reference to the  density of  $M_{\alpha;z}$.
\end{remark}
Beyond Janson's proof, replacing  intractable integrals with easier  double integrals is a recurring theme in this paper. 
We  use it in an alternate proof, due to Feller,
of $\mathbb{E}\left[e^{-x M_{\alpha;z}}\right]$ 
based  on  the explicit density of $M_{\alpha;z}$ without reference to its moments.
The benefit of this  proof  is that we may extend it  to  stable variable ratios and powers where 
densities can be found but moments are not readily available or may not exist.
Throughout, \(\Pr(T\equiv t)\) denotes the density of a continuous random variable \(T\)  at $t$.

\begin{proposition}
\label{prop:stablePowerDensity}
 Let $S^{-\beta}_{\alpha;z}$  be a stable variable power for $\beta>0$ with $M_{\alpha;z} \equiv S^{-\alpha}_{\alpha;z}$. Then 
\begin{align}
\Pr(S^{-\beta}_{\alpha;z} \equiv t) &= \frac{1}{\beta t}  f_\alpha(t^{-1/\beta}\vert z) \, t^{-1/\beta}
\equiv \frac{1}{\beta t} f_\alpha(1\vert z t^{\alpha/\beta})    
 \label{eq:stablePowerDensity} \\
 \Pr(M_{\alpha;z} \equiv t) &= \frac{1}{\alpha t}  f_\alpha(t^{-1/\alpha}\vert z) \, t^{-1/\alpha} \equiv \frac{1}{\alpha t} f_\alpha(1\vert z t)    
 \label{eq:MLDensity} 
\end{align}
The Mittag-Leffler variable $M_{\alpha;z}$ has Laplace transform 
\begin{align}
\mathbb{E}\left[e^{-x M_{\alpha;z}}\right] &\equiv 
\frac{1}{\alpha} \int_0^\infty e^{-xt} \,   f_\alpha(1 \vert zt) \, \frac{dt}{t} = E_\alpha(-x/z) 
 \label{eq:MLLT1}
\end{align}
\end{proposition}
\begin{proof}[Proof of Proposition~$\ref{prop:stablePowerDensity}$] 
\label{proof:stablePowerDensity}
$\Pr(S_{\alpha;z} \equiv u) = f_\alpha(u\vert z)$ and $\Pr(S^{-\beta}_{\alpha;z}  \equiv t \vert u) = \delta(t-u^{-\beta})$, hence
\begin{align*}
\Pr(S^{-\beta}_{\alpha;z} \equiv t) &= 
\int_0^\infty \delta(t-u^{-\beta}) f_\alpha(u\vert z) du \\
 &= \frac{1}{\beta t} f_\alpha(t^{-1/\beta}\vert z) t^{-1/\beta} 
  \equiv \frac{1}{\beta t} f_\alpha(1\vert z t^{\alpha/\beta})
\end{align*}
 Feller~\cite[XIII.8(b)]{Feller2} gave a proof  of~(\ref{eq:MLLT1})  equivalent to the following.
Let $u f_\alpha(u \vert zt)/\alpha t$ be a two-dimensional density on the positive quadrant $(u>0,t>0)$.
First take the Laplace  transform over $u$, then over $t$
\begin{align*}
\frac{1}{\alpha t} \int_0^\infty e^{-su} u f_\alpha(u \vert zt) du 
 &= -\frac{1}{\alpha t} \frac{d}{ds} e^{-zt s^\alpha} = zs^{\alpha-1}e^{-zt s^\alpha} \\
\textrm{then\quad} zs^{\alpha-1} \int_0^\infty e^{-x t} e^{-zt s^\alpha} dt
   &= \frac{zs^{\alpha-1}}{x+z s^\alpha}
     \equiv \int_0^\infty e^{-su} E_\alpha(-x u^\alpha/z) du
 \intertext{By reversing the order of integration, we infer that}    
\frac{u}{\alpha} \int_0^\infty e^{-xt} f_\alpha(u \vert zt) \frac{dt}{t} &= E_\alpha(-x u^\alpha/z) \\
\implies  \frac{1}{\alpha} \int_0^\infty e^{-xt} f_\alpha(1 \vert zt) \frac{dt}{t} &= E_\alpha(-x/z) \quad(u=1)
 \end{align*}
 thereby proving  $\mathbb{E}\left[e^{-x M_{\alpha;z}}\right]= E_\alpha(-x/z)$ 
 without a need for moments $\mathbb{E}\bigl[M^k_{\alpha;z} \bigr]$.
\end{proof}


 \begin{proposition}
\label{prop:infseries}
The stable density $f_\alpha(t\vert  z)$ has an infinite series representation
 \begin{align}
  f_\alpha(t\vert z)  &= \frac{1}{\pi}\,{\rm Im}  \int_0^\infty e^{-tv} \, e^{-z (e^{-i\pi}v)^\alpha}  \, dv  
  \label{eq:intStieltjesInverse} \\
&= \phantom{-} \frac{1}{\pi}\,{\rm Im} \sum_{k=0}^\infty \frac{(-z)^k}{k!} e^{-i\pi\alpha k}  \int_0^\infty e^{-tv} \, v^{\alpha k} \, dv
\label{eq:Pollardinfsum1} \\
&= -\frac{1}{\pi}
   \sum_{k=1}^\infty \frac{(-z)^k}{k!} \sin(\pi\alpha k) \, \frac{\Gamma(\alpha k+1)}{t^{\alpha k+1}}
\label{eq:PollardStable} 
\end{align}
\end{proposition}
\begin{proof}[Proof of Proposition~$\ref{prop:infseries}$] 
\label{proof:infseries}
See Pollard~\cite{Pollard} and Feller~\cite[XVII.6]{Feller2}.
In brief, $(\ref{eq:intStieltjesInverse})$ 
 expresses  $f_\alpha(t\vert z)$ as the Laplace transform of the Stieltjes 
inverse of its  Laplace transform $e^{-z x^\alpha} $. 
$(\ref{eq:Pollardinfsum1})$ and $(\ref{eq:PollardStable})$  follow from Taylor expansion of  $\exp(-z (e^{-i\pi}v)^\alpha)$.
\end{proof}


\section{Main Contribution}
\label{sec:contribution}

 \begin{theorem}
\label{thm:gen_rvproduct}
Let $S_{\alpha;z}$ be the stable random variable of Definition~\ref{def:stable} and let 
 $U \sim F$ be a positive random variable with probability distribution  $F$ on the positive half-line,
 where $S_{\alpha;z}$ and $U$ are independent.
Then the product $S_{\alpha;z} \, U^{1/\alpha}$  has a density
\begin{align}
\Pr(S_{\alpha;z} \, U^{1/\alpha} \equiv t) &\equiv 
h_\alpha(t \vert z, F)
=  \int_0^\infty  f_\alpha(t \vert z u) \, dF(u)  
\label{eq:gen_rvproduct_density}
\intertext{and Laplace transform
 $\mathbb{E}\left[e^{-x S_{\alpha;z} U^{1/\alpha}}\right] =  \mathbb{E}\left[e^{-zx^\alpha U}\right] \; (x\ge0)$. 
The power $(S_{\alpha;z}\,U^{1/\alpha})^{-\beta}$ for $\beta>0$  has a density}
\Pr((S_{\alpha;z}\,U^{1/\alpha})^{-\beta}\equiv t)
  &= \frac{1}{\beta t} h_\alpha(1\vert zt^{\alpha/\beta}, F)
    =  \frac{1}{\beta t}  \int_0^\infty  f_\alpha(1 \vert zt^{\alpha/\beta} u) \, dF(u)  
\label{eq:gen_rvproductPowerDensity}
\end{align}
\end{theorem}


\begin{corollary}
\label{cor:stableMLproduct}
Let  $U=M_\sigma \sim {\rm ML}(\sigma)$ $(0<\sigma<1)$ have Mittag-Leffler  
distribution ${\rm ML}(\sigma)$.
Then $S_{\alpha;z}  M^{1/\alpha}_\sigma$  
has a density $h_{\alpha, \sigma}(t \vert z) \equiv h_\alpha(t\vert z, {\rm ML}(\sigma))$
\begin{align}
\Pr(S_{\alpha;z}  M^{1/\alpha}_\sigma \equiv t)
&= h_{\alpha, \sigma}(t \vert z)
  =  \int_0^\infty  f_\alpha(t \vert z u) \, d{\rm ML}(u \vert \sigma) 
\label{eq:stableMLproductDensity}
\intertext{and Laplace transform 
 $\mathbb{E}\left[e^{-x S_{\alpha;z} M^{1/\alpha}_\sigma} \right] =  \mathbb{E}\left[e^{-zx^\alpha M_\sigma}\right] 
 =E_\sigma(-zx^\alpha) \; (x\ge0)$. 
 For $\beta>0$} 
 \Pr((S_{\alpha;z}  M^{1/\alpha}_\sigma)^{-\beta} \equiv t)  
 &= \frac{1}{\beta t} h_{\alpha, \sigma}(1 \vert zt^{\alpha/\beta}) \quad (\beta>0)
\label{eq:stableMLproductPowerdensity}
\end{align}
 For $\sigma=\alpha$, $S_{\alpha;z}  M^{1/\alpha}_\alpha \equiv S_{\alpha;z}/ S_\alpha$ 
 and inserting  the  representation of Proposition~$\ref{prop:infseries}$ for $f_\alpha(t\vert zu)$ 
 in~$(\ref{eq:stableMLproductDensity})$ 
 reduces $h_{\alpha, \alpha}(t \vert z)$  to a simple geometric series 
 \begin{align}
 h_{\alpha,\alpha}(t \vert z)
 &= \frac{1}{\pi t}\,  {\rm Im} \, \frac{1}{1+z e^{-i\pi\alpha} t^{-\alpha}} 
 =  \frac{\sin\pi\alpha}{\pi}\, \frac{z\, t^{\alpha-1}} {z^2+2z \, t^\alpha \cos\pi\alpha+t^{2\alpha}} 
 \label{eq:hGeometricSeries}
\end{align}
\end{corollary}

\begin{proof}[Proof of Theorem~$\ref{thm:gen_rvproduct}$]
\label{proof:gen_rvproduct}
The representation-free densities in Theorem~\ref{thm:gen_rvproduct} 
arise readily from the  independence  of $S_{\alpha;z}$ and $U$. 
For $\sigma=\alpha$,  the infinite series~(\ref{eq:Pollardinfsum1}) leads to
\begin{align*}
 h_{\alpha,\alpha}(t \vert z)
 &= -  \frac{1}{\pi}\,   \sum_{k=1}^\infty  \frac{(-z)^k}{k!} \sin(\pi\alpha k) \; \frac{\Gamma(\alpha k+1)}{t^{\alpha k+1}}  \; 
       \int_0^\infty u^k \, d{\rm ML}(u \vert \alpha) \\
 &=  \frac{1}{\pi t}\,  {\rm Im} \,\sum_{k=0}^\infty  (-z e^{-i\pi\alpha} t^{-\alpha})^k  
 \qquad  (\mathbb{E}\bigl[M^k_\alpha\bigr] = k!/\Gamma(\alpha k+1))  \\
 &= \frac{1}{\pi t}\,  {\rm Im} \, \frac{1}{1+z e^{-i\pi\alpha} t^{-\alpha}}  
 =  \frac{\sin\pi\alpha}{\pi}\, \frac{z\, t^{\alpha-1}} {z^2+2z \, t^\alpha \cos\pi\alpha+t^{2\alpha}} 
\end{align*}
which proves~(\ref{eq:hGeometricSeries}) (having used term-by-term integration).
\end{proof}
\begin{remark}
\label{rem:theorem}
The subtlety here is not in the general Theorem~$\ref{thm:gen_rvproduct}$ in itself, but in the derivation of  
explicit Laplace and/or Stieltjes  transforms $\mathbb{E}\bigl[e^{-x T}\bigr]$  
and $\mathbb{E}\bigl[1/(s+T)\bigr]$ respectively for a wide range of  ratios $T$
under Corollary~\ref{cor:stableMLproduct}. 
Accordingly, Table~\ref{table:rvFreeDensities} gives representation-free densities as well as Laplace and Stieltjes transforms  
of various random variable ratios based on Corollary~\ref{cor:stableMLproduct}.
\end{remark}

\begin{table}[h!]
{\small
\setcounter{rownum}{0}
\begin{center}
\begin{tabular}{| c | c  | c | c | c |} 
\cline{2-5} 
\multicolumn{1}{c|}{} &  \multicolumn{1}{ c |}{ r.\,v.  \;$T$ } &
\multicolumn{1}{ c |}{Representation-free density $\Pr(T)$ } & \multicolumn{1}{c |} {
\begin{tabular}{c}
Laplace \\ $\mathbb{E}\bigl[e^{-xT} \bigr]$
\end{tabular}}  
& \multicolumn{1}{c |} {
\begin{tabular}{c}
Stieltjes \\ $\mathbb{E}\bigl[ \frac{1}{s+T} \bigr]$
\end{tabular}} 
\TBstrut\\
\cline{2-5} 
  \multicolumn{5}{ c }{$0<\alpha,\sigma<1;\;  \beta>0, z>0$}  
 \Tstrut \\
\hline 
\refstepcounter{rownum}{\scriptsize\therownum}\label{free:stable} &
$S_{\alpha;z}$ &  $f_\alpha(t \vert z)$ & $\exp(-zx^\alpha)$ &  \TBstrut\\ 
\refstepcounter{rownum}{\scriptsize\therownum}\label{free:stablePower} &
$S^{-\beta}_{\alpha;z}$ & 
\(\displaystyle\frac{1}{\beta t} \, f_\alpha(1 \vert z t^{\alpha/\beta})\)  &  & \Bstrut\\  
\refstepcounter{rownum}{\scriptsize\therownum}\label{free:mittag} &
\(M_{\alpha;z} \equiv S^{-\alpha}_{\alpha;z}\) & 
\(\dfrac{1}{\alpha t} \, f_\alpha(1 \vert  z t)\) 
& $E_{\alpha}(- x/z)$  & \Bstrut\\ 
 \hline 
\hline 
\refstepcounter{rownum}{\scriptsize\therownum}\label{free:stableRatio1} &
 \( \displaystyle 
 S_{\alpha;z}  M^{1/\alpha}_\sigma\)  
& \( \displaystyle h_{\alpha,\sigma}(t \vert z) = \int_0^\infty  f_\alpha(t \vert z u) \, d{\rm ML}(u\vert \sigma) \) 
& \( E_{\sigma}(- zx^\alpha) \)   & \TBstrut\\ 
\refstepcounter{rownum}{\scriptsize\therownum}\label{free:stableRatioPower1} &
\( \displaystyle (S_{\alpha;z}  M^{1/\alpha}_\sigma)^{-\beta} \)
& \( \displaystyle \frac{1}{\beta t} \, h_{\alpha,\sigma}(1 \vert z t^{\alpha/\beta}) \)
 & & \TBstrut\\ 
\hline 
  \multicolumn{5}{ c }{$\sigma=\alpha$}  
 \Tstrut \\
\hline 
\refstepcounter{rownum}{\scriptsize\therownum}\label{free:stableRatio2} &
\( \displaystyle S_{\alpha;z}  M^{1/\alpha}_\alpha \equiv\frac{S_{\alpha;z}}{S_{\alpha}} \)
& \( h_{\alpha,\alpha}(t \vert z) \)
& \( E_{\alpha}(- zx^\alpha) \)
&  \( \displaystyle \frac{s^{\alpha-1}}{z+s^\alpha} \)
\TBstrut\\ 
\refstepcounter{rownum}{\scriptsize\therownum}\label{free:stableRatioPower2} 
& \( \displaystyle \left(\frac{S_{\alpha;z}}{S_{\alpha}}\right)^{\!\!-\beta} \) 
& \( \displaystyle  \frac{1}{\beta t} \, h_{\alpha,\alpha}(1 \vert zt^{\alpha/\beta}) \)
&  & \TBstrut\\ 
\refstepcounter{rownum}{\scriptsize\therownum}\label{free:mittagRatio} &
 \( \displaystyle \frac{M_{\alpha}}{ M^\prime_{\alpha}} \equiv \left(\frac{S_{\alpha}}{S^\prime_\alpha}\right)^{\!\!-\alpha} \)
& \( \displaystyle  \frac{1}{\alpha t} h_{\alpha,\alpha}(1 \vert t) \)
&  &  \TBstrut\\ 
\refstepcounter{rownum}{\scriptsize\therownum}\label{free:stableRatioInv} &
\( \dfrac{S_{\alpha}}{S_{\alpha;z}} \)
& \( \displaystyle \frac{1}{t} \, h_{\alpha,\alpha}(1 \vert zt^{\alpha}) \)
& \( E_{\alpha}(- x^\alpha/z) \)
&  \( \displaystyle \frac{zs^{\alpha-1}}{1+zs^\alpha} \)
\TBstrut\\ 
\hline 
 \multicolumn{2}{ c  }{} & \multicolumn{1}{ c  | }{$\beta=(1-\alpha)/\alpha$} 
& \multicolumn{2}{c |} {
\begin{tabular}{c}
{\scriptsize  Analytically continued}  \\
{\scriptsize substitute for $\mathbb{E}\bigl[\frac{1}{s+T} \bigr]$}
\end{tabular} }
\TBstrut\\
\hline 
\refstepcounter{rownum}{\scriptsize\therownum}\label{free:stableRatioPower3} 
&  
\begin{tabular}{c}
\( \displaystyle \left(\frac{S_{1-\alpha;z}}{S_{1-\alpha}}\right)^{\!\!-\beta} \equiv \) \\
 \( \left(\dfrac{M_{1-\alpha;z}}{ M_{1-\alpha}}\right)^{\!\!-1/\alpha} \)
\end{tabular}  
&  \( \displaystyle \frac{1}{\beta t} \, h_{1-\alpha,1-\alpha}(1 \vert z t^{\alpha}) \)
&   \multicolumn{2}{ c |}{ \( \displaystyle  \frac{1}{\beta} \frac{zs^{\alpha-1}}{1-zs^\alpha} \)} 
\TBstrut\\ 
\hline
\multicolumn{2}{c|}{} 
& \multicolumn{1}{ c |}{$\Pr(U=1/(1+T) \in [0,1])$}  
& \multicolumn{2}{c |} {Stieltjes: $\mathbb{E}\bigl[ \frac{1}{s+U} \bigr]$} 
\TBstrut\\
\hline 
\refstepcounter{rownum}{\scriptsize\therownum}\label{free:stableRatio[0,1]} 
& \( \displaystyle \frac{S_{\alpha}}{S_{\alpha}+S_{\alpha;z}}  \) 
&  \begin{tabular}{c}
\(  \displaystyle  \phantom{\frac{1}{\beta}}  \int_0^{\infty} \delta\bigl(u-\tfrac{1}{1+t}\bigr)  h_{\alpha,\alpha}(t \vert z) dt \) 
\Tstrut\\
\( \displaystyle \quad \equiv \frac{1}{u^2} \, h_{\alpha,\alpha}\bigl(\tfrac{1-u}{u} \vert z\bigr) \) 
\Bstrut
\end{tabular} 
& \multicolumn{2}{ c |}{ \( \displaystyle \frac{zs^{\alpha-1}+(1+s)^{\alpha-1}}{zs^\alpha+(1+s)^\alpha} \) }
\TBstrut\\
\refstepcounter{rownum}{\scriptsize\therownum}\label{free:stableRatioPower[0,1]} 
& \( \displaystyle \frac{S^{-\beta}_{1-\alpha}}{S^{-\beta}_{1-\alpha}+ S^{-\beta}_{1-\alpha;z}} \)
&   \begin{tabular}{c}
\( \displaystyle \frac{1}{\beta} \int_0^\infty \delta\bigl(u-\tfrac{1}{1+t}\bigr) h_{1-\alpha,1-\alpha}(1 \vert z t^{\alpha}) \frac{dt}{t} \) 
\Bstrut\\
\( \displaystyle  \equiv \frac{1}{\beta} \frac{1}{u(1-u)} \, h_{1-\alpha,1-\alpha}\bigl(1 \vert z \bigl(\tfrac{1-u}{u}\bigr)^\alpha \bigr)  \)
\end{tabular} 
 &  \multicolumn{2}{ c |}{ \( \displaystyle  \frac{1}{\beta}  \frac{z(1+s)^{\alpha-1}-s^{\alpha-1}}{s^\alpha-z(1+s)^\alpha} \) } 
\TBstrut\\
\refstepcounter{rownum}{\scriptsize\therownum}\label{free:mittagRatioPower[0,1]} 
&\(  \displaystyle \frac{M^{\prime 1/\alpha}_{1-\alpha}}{M^{\prime 1/\alpha}_{1-\alpha}+ M^{1/\alpha}_{1-\alpha}} \)
&   \multicolumn{1}{ l |}{
\( \displaystyle \frac{1}{\beta} \int_0^\infty \delta\bigl(u-\tfrac{1}{1+t}\bigr) h_{1-\alpha,1-\alpha}(1 \vert t^{\alpha}) \frac{dt}{t} \) }
\TBstrut
 &  \multicolumn{2}{ c |}{ \( \displaystyle \frac{1}{\beta}  \frac{(1+s)^{\alpha-1}-s^{\alpha-1}}{s^\alpha-(1+s)^\alpha} \) } 
\strut\\
\hline
\end{tabular}
\end{center}
} 
\caption{Random variables,   representation-free  densities, Laplace/Stieltjes transforms}
\label{table:rvFreeDensities}
\end{table}

\begin{remark}
\label{rem:analyticContinuation}
  Table\ref{table:rvFreeDensities}[row\ref{free:stableRatioPower3}] ,
arising as it does from analytic continuation of the gamma function (see Proposition~\ref{prop:stableRatioPowerDensity}),
is not, as it stands, the Stieltjes transform of a probability density.
It is formally  the Laplace transform of $-E_\alpha(x^\alpha/z)$, which  is  not completely monotone (or even positive), 
as the Laplace transform of a density should be by Bernstein's theorem.
We may nonetheless use  [row\ref{free:stableRatioPower3}]  as an intermediate step  to derive the   Stietljes transform 
of [row\ref{free:stableRatioPower[0,1]}],  which is well-behaved on $s\ge0$. 
\end{remark}

\begin{table}[h!]
\setcounter{rownum}{5}
{\small
\begin{center}
\begin{tabular}{| c | c  | l | c | c |} 
\cline{2-5} 
\multicolumn{1}{c|}{} &  \multicolumn{1}{ c |}{ r.\,v.  \;$T$ } &
\multicolumn{1}{ c |}{Simple explicit  density $\Pr(T)$ } & \multicolumn{1}{c |} {
\begin{tabular}{c}
Laplace \\ $\mathbb{E}\bigl[e^{-xT} \bigr]$
\end{tabular}}  
& \multicolumn{1}{c |} {
\begin{tabular}{c}
Stieltjes \\ $\mathbb{E}\bigl[ \frac{1}{s+T} \bigr]$
\end{tabular}} 
\TBstrut\\
\cline{2-5} 
  \multicolumn{5}{ c }{Series form~(\ref{eq:PollardStable}) of stable density in 
    Table~\ref{table:rvFreeDensities} for $\sigma=\alpha$ gives simple densities}  
 \Tstrut \\
\hline 
\refstepcounter{rownum}{\scriptsize\therownum}\label{simple:stableRatio2} 
 & \( \displaystyle 
         \frac{S_{\alpha;z}}{S_{\alpha}} \)  
& \( \displaystyle \frac{\sin\pi\alpha}{\pi}\, \frac{z\, t^{\alpha-1}} {z^2+2z \, t^\alpha \cos\pi\alpha+t^{2\alpha}}  \)
& \( E_{\alpha}(- zx^\alpha) \)
&  \( \displaystyle \frac{s^{\alpha-1}}{z+s^\alpha} \)
\TBstrut\\ 
\refstepcounter{rownum}{\scriptsize\therownum}\label{simple:stableRatioPower2} 
& \( \displaystyle \left(\frac{S_{\alpha;z}}{S_\alpha}\right)^{\!\!-\beta}  \)
& \( \displaystyle   
 \frac{\sin\pi\alpha}{\pi\beta}\, \frac{z\, t^{\alpha/\beta-1}} {z^2t^{2\alpha/\beta}+2z \, t^{\alpha/\beta} \cos\pi\alpha+1} \)    
 & & \TBstrut\\ 
\refstepcounter{rownum}{\scriptsize\therownum}\label{simple:mittagRatio} 
& \( \displaystyle \frac{M_\alpha}{ M^\prime_\alpha} \equiv \left(\frac{S_{\alpha}}{S^\prime_\alpha}\right)^{\!\!-\alpha} \)
& \( \displaystyle \frac{\sin\pi\alpha}{\pi \alpha}\, \frac{1} {t^2+2\, t \cos\pi\alpha+1} \)
 &   & \TBstrut\\ 
\refstepcounter{rownum}{\scriptsize\therownum}\label{simple:stableRatioInv} 
& \( \dfrac{S_{\alpha}}{S_{\alpha;z}} \)
& \( \displaystyle \frac{\sin\pi\alpha}{\pi}\, \frac{z\, t^{\alpha-1}} {z^2t^{2\alpha}+2z \, t^\alpha \cos\pi\alpha+1} \)
& \( E_{\alpha}(- x^\alpha/z) \)
&  \( \displaystyle \frac{zs^{\alpha-1}}{1+zs^\alpha} \) \TBstrut\\ 
\hline 

 \multicolumn{2}{ c  }{} 
& \multicolumn{1}{ c  | }{$\beta=(1-\alpha)/\alpha$} 
& \multicolumn{2}{c |} {
\begin{tabular}{c}
{\scriptsize  Analytically continued}  \\
{\scriptsize substitute for $\mathbb{E}\bigl[\frac{1}{s+T} \bigr]$}
\end{tabular} }
\TBstrut\\
\hline 
\refstepcounter{rownum}{\scriptsize\therownum}\label{simple:stableRatioPower3} 
& \begin{tabular}{c}
\( \displaystyle \left(\frac{S_{1-\alpha;z}}{S_{1-\alpha}}\right)^{\!\!-\beta} \equiv \) \\
 \( \left(\dfrac{M_{1-\alpha;z}}{ M_{1-\alpha}}\right)^{\!\!-1/\alpha} \)
\end{tabular}  & \( \displaystyle \frac{\sin\pi\alpha}{\pi\beta}\, \frac{z\, t^{\alpha-1}} {z^2t^{2\alpha}-2z \, t^{\alpha} \cos\pi\alpha+1}   \)
&  \multicolumn{2}{ c |}{ \( \displaystyle  \frac{1}{\beta} \frac{zs^{\alpha-1}}{1-zs^\alpha} \)}  
\TBstrut\\ 
\hline
\multicolumn{2}{c|}{} 
& \multicolumn{1}{ c |}{$\Pr(U=1/(1+T) \in [0,1])$}  
& \multicolumn{2}{c |} {Stieltjes: $\mathbb{E}\bigl[ \frac{1}{s+U} \bigr]$}  
\Tstrut\\
\hline
&   \multicolumn{1}{ l |}{} 
&  \multicolumn{1}{ l |}{$\bar{u}=1-u$} 
& \multicolumn{2}{ c |}{ } \\
\refstepcounter{rownum}{\scriptsize\therownum}\label{simple:stableRatio[0,1]} 
& \( \displaystyle \frac{S_{\alpha}}{S_\alpha+S_{\alpha;z}} \)
& \( \displaystyle  \frac{\sin\pi\alpha}{\pi}\, 
\frac{zu^{\alpha-1} \bar{u}^{\alpha-1}} {z^2 u^{2\alpha}+2z u^\alpha \bar{u}^\alpha  \cos\pi\alpha+\bar{u}^{2\alpha}} \)
& \multicolumn{2}{ c |}{ \( \displaystyle \frac{zs^{\alpha-1}+(1+s)^{\alpha-1}}{zs^\alpha+(1+s)^\alpha } \) }
\TBstrut\\
\refstepcounter{rownum}{\scriptsize\therownum}\label{simple:stableRatioPower[0,1]} 
& \( \displaystyle \frac{S^{-\beta}_{1-\alpha}}{S^{-\beta}_{1-\alpha}+ S^{-\beta}_{1-\alpha;z}} \)
& \( \displaystyle \frac{\sin\pi\alpha}{\pi \beta} 
 \frac{z\bar{u}^{\alpha-1} u^{\alpha-1}} {u^{2\alpha}-2z \bar{u}^\alpha u^\alpha \cos\pi\alpha+z^2\bar{u}^{2\alpha}} \)
\TBstrut
 &  \multicolumn{2}{ c |}{ \( \displaystyle \frac{1}{\beta}  \frac{z(1+s)^{\alpha-1}-s^{\alpha-1}}{s^\alpha-z(1+s)^\alpha} \) } 
\TBstrut\\
\refstepcounter{rownum}{\scriptsize\therownum}\label{simple:mittagRatioPower[0,1]} 
& \( \displaystyle \frac{M^{\prime 1/\alpha}_{1-\alpha}}{M^{\prime 1/\alpha}_{1-\alpha}+ M^{1/\alpha}_{1-\alpha}} \)
& \( \displaystyle   \frac{\sin\pi\alpha}{\pi \beta} 
 \frac{\bar{u}^{\alpha-1} u^{\alpha-1}} {u^{2\alpha}-2 \bar{u}^\alpha u^\alpha \cos\pi\alpha+\bar{u}^{2\alpha}} \)
\TBstrut
 &  \multicolumn{2}{ c |}{ \( \displaystyle \frac{1}{\beta}  \frac{(1+s)^{\alpha-1}-s^{\alpha-1}}{s^\alpha-(1+s)^\alpha}  \) }
\TBstrut\\
\hline
\hline
\refstepcounter{rownum}{\scriptsize\therownum}\label{simple:stableRatioBetaDensity} 
& \( \displaystyle \frac{S_{1/2}}{S_{1/2}+S^\prime_{1/2}} \)
& \multicolumn{1}{ c |}{\( \displaystyle  {\rm beta}\bigl(u \vert \tfrac{1}{2}, \tfrac{1}{2}\bigr) =  \frac{1}{\pi}\, \frac{1}{\sqrt{u(1-u)}} \) }
& 
\multicolumn{2}{ c |}{ \( \displaystyle  \frac{1}{\sqrt{s(1+s)}} \) }
\TBstrut\\
\hline
\end{tabular}
\end{center}
} 
\caption{Random variables, simple densities, Laplace/Stieltjes transforms}
\label{table:rvSimpleDensities}
\end{table}
Table~\ref{table:rvSimpleDensities} replicates the entries of 
Table\ref{table:rvFreeDensities}[rows~\ref{free:stableRatio2}-\ref{free:mittagRatioPower[0,1]}] $(\sigma=\alpha)$
with the  representation-free  densities replaced by the more familiar simple densities induced  by substituting 
the geometric series~$(\ref{eq:hGeometricSeries})$ for $h_{\alpha,\alpha}(t \vert z)$.
 The  densities of 
 Table\ref{table:rvSimpleDensities}[rows~\ref{simple:stableRatio2},\ref{simple:mittagRatio},\ref{simple:stableRatioInv}]
 can be traced to Chaumont and Yor~{\rm \cite[4.23.3]{ChaumontYor}}.
We are not aware of literature on the more general density of $(S_{\alpha;z}/S_\alpha)^{-\beta}$ of 
[row~\ref{simple:stableRatioPower2}], from which [rows~\ref{simple:mittagRatio},\ref{simple:stableRatioInv}]  
follow for $\{\beta=\alpha,1\}$ respectively.

\begin{remark}
\label{rem:inversionfree}
None of the simple densities in Table~\ref{table:rvSimpleDensities}   arise from direct  inversion. 
We merely  substitute the geometric series~$(\ref{eq:hGeometricSeries})$ for $h_{\alpha,\alpha}(t \vert z)$
which arises, in turn, from the infinite series representation~(\ref{eq:PollardStable}) of the stable density. 
Inversion is  thereby  confined  to the manner  in which~(\ref{eq:PollardStable}) was derived.
Hence we obtain the simple densities of [rows~\ref{simple:stableRatioPower2},\ref{simple:mittagRatio}]
even though we have not stated their Laplace or Stieltjes transforms.
Where the Stieltjes transform is available  by some other  construction, 
its  inverse may (after  verifying nonnegativity) yield a simple density,
 but not its representation-free variant. 
 \end{remark}

\subsection{Lamperti Law}
\label{sec:lamperti}
The density of Table\ref{table:rvSimpleDensities}[row~\ref{simple:stableRatio[0,1]}]
 and its Stieltjes transform are  of particular interest.
For $z=(1-\alpha)/\alpha$, they exactly reproduce expressions  in Lamperti~\cite[(1.4) and  (3.17)]{Lamperti},
obtained in a study of  occupation times for a class of stochastic processes.
Without reference to the stable density,  
Lamperti used a theorem on slowly varying functions together with Tauberian arguments  to derive the Stieltjes transform, 
hence the density by Stieltjes inversion.
Starting from the density as the primary object did  not arise.

\subsection{BFRY Law}
\label{sec:bfry}
All entries of Table\ref{table:rvSimpleDensities}[row~\ref{simple:mittagRatioPower[0,1]}] arise in Bertoin~{\it et al.}~\cite{Bertoin} (BFRY) on the distributional properties of the excursion duration straddling an independent exponential time of a recurrent 
Bessel process of  dimension $d=2(1-\alpha)$ for  $0<d<2$ or $0<\alpha<1$.
BFRY initially derived the density by Stieltjes inversion, as in the Lamperti case,
but the Stieltjes transform arose from the GGC  construction discussed in Section~\ref{sec:GGC}.
The Lamperti and BFRY constructions  coincide  exactly for dimension $d=1$ or $\alpha=1/2$ of
 Table\ref{table:rvSimpleDensities}[row\ref{simple:stableRatioBetaDensity}] 
with density ${\rm beta}\bigl(u \vert \tfrac{1}{2}, \tfrac{1}{2}\bigr)$
and Stieltjes transform $1/\sqrt{s(1+s)}$.

Noting that 
inversion ``may seem quite unnatural'', BFRY turned to a   "more intuitive" derivation of the density.  
Symbolically,
explicit inversion gives the  density ${\cal S}^{-1} F_\alpha$ where ${\cal S}$ is the Stieltjes operator 
and $F_\alpha$  the known Stieltjes transform.  
An alternate view  is to solve ${\cal S}d_\alpha=F_\alpha$
for the density  $d_\alpha$ without  resorting to  explicit inversion.
The typical approach to such  inverse problems is to construct  a parametric model for $d_\alpha$ 
and solve for  the  model parameters.
Inspired by the $\alpha=1/2$ case,
 Bertoin~{\it et al.}~\cite[2.2.3]{Bertoin}  proposed a beta mixture model 
$d_\alpha(u)= \int {\rm beta}(u\vert \gamma,1-\gamma) \mu_\alpha(d\gamma)$
and  searched for the measure $\mu_\alpha$ (model "parameters")  such that ${\cal S}d_\alpha=F_\alpha$.
They concluded that a  rigorous argument  relied 
on ``holomorphic prolongation" of  ${\rm beta}(u\vert \gamma,1-\gamma)$  beyond $\gamma\in[0,1]$.

In our case, we do not know $F_\alpha$ {\it a priori}.
Instead, we start with a representation-free density $d_\alpha$ induced by a specified stable variable ratio 
and find the Stieltjes transform $F_\alpha= {\cal S}d_\alpha$
without a need to contemplate  a  beta mixture or any other  representation of  $d_\alpha$.
As we show in the proofs of Section~\ref{sec:proofs}, 
the representation-free Lamperti  density readily leads to  the corresponding forward  Stieltjes transform.
The corresponding forward transform calculation for the representation-free BFRY density follows the same general strategy, 
but requires the additional intermediate step of analytic continuation of the gamma function.
While our  representation-free forward transform objective  $F_\alpha= {\cal S}d_\alpha$ 
 contrasts with the BFRY  inverse problem objective ${\cal S}d_\alpha=F_\alpha$ (without direct inversion),
the two objectives share a common need for  analytic continuation. 

 \section{Generalised Gamma Convolutions}
 \label{sec:GGC}
 
  The BFRY example above  arose in a  study of a  self-decomposable (SD) distribution. 
  In fact, as BFRY noted, the latter  is an instance of a  sub-class of SD distributions known as  
  generalised  gamma convolutions (GGCs) introduced by Thorin~\cite{Thorin}
 and studied in detail by Bondesson~\cite{Bondesson}.
 We give  a graphical representation of GGCs 
 and present  some results that complement BFRY and the associated discussion in 
 James~{\it et al.}~\cite{JamesRoynetteYor} (JRY).
  
\begin{table}[h!]
\setcounter{rownum}{0}
{\normalsize
\begin{center}
\begin{tabular}{| c | c | c | c |} 
\hline
& 
& \( f(x \vert \mu) \)
& \( \exp(-\mu \psi(s)) \)
\TBstrut\\
\refstepcounter{rownum}{\scriptsize\therownum}\label{genGGC} 
&   &   $\big\downarrow_\ell$  & $\big\downarrow_d$  $\big\uparrow^u$ \\
&  \( \displaystyle  \mu \, \tau(t) \)
&  \( \displaystyle  \mu \, \rho(x) \)
&  \(\displaystyle \mu\, \psi^\prime(s) \)
\TBstrut\\
\hline
 \multicolumn{4}{  l } {{\small  $\uparrow^u: \phantom{\mu}\psi^\prime(s) = \int_0^\infty e^{-sx}  \rho(x) dx
\implies  \psi(s) = \int_0^\infty \left(1-e^{-sx}\right)  \rho(x) \frac{dx}{x}$} {\small (L\'{e}vy-Khintchine)}}  \\
  \multicolumn{4}{  l } {{\small $\downarrow_d:  
      \mu\psi^\prime(s) = -\phi^\prime_\mu(s)/\phi_\mu(s) \quad (\phi_\mu(s)=\exp(-\mu \psi(s))$}} \\
  \multicolumn{4}{  l } {{\small $\downarrow_\ell:  
    \mu\rho(x) \;= \lim_{n\to\infty} nxf(x\vert \tfrac{\mu}{n})$ \, (we omit  compound Poisson  $\mu\rho(x)\to f(x\vert \mu))$ }}

\Bstrut\\
\hline
& 
& \( f_\alpha(x \vert z) \)
& \( \exp(-z s^\alpha) \)
\TBstrut\\
\refstepcounter{rownum}{\scriptsize\therownum}\label{stableGGC} 
&   &   $\big\downarrow$  & $\big\downarrow$  $\big\uparrow$ \\
& \( \displaystyle  z\, \tau_\alpha(t) \equiv z \frac{\alpha \sin\pi\alpha}{\pi} \, t^{\alpha-1} \)
& \( \displaystyle  z \rho_\alpha(x) \equiv   \frac{z \alpha}{\Gamma(1-\alpha)} \, x^{-\alpha} \)
& \( z  \alpha s^{\alpha-1} \)
\TBstrut\\
\hline
   \multicolumn{1}{ c } {} 
& \multicolumn{3}{  l } {{\scriptsize $u=1/(1+t); \, \bar{u}=1-u; \,  \beta=\alpha/(1-\alpha)$}} \strut\\
\hline
& 
&  & \( \displaystyle \bigl((1+s)^{\alpha}-s^{\alpha}\bigr)^{\mu/(1-\alpha)} \)
\TBstrut\\
\refstepcounter{rownum}{\scriptsize\therownum}\label{genBFRYGGC} 
&   &     & $\big\downarrow$  $\big\uparrow$  \Bstrut\Bstrut\\
&  \( \displaystyle \frac{\mu}{\beta u\bar{u}} \, h_{1-\alpha,1-\alpha}\bigl(1 \vert  \bigl(\tfrac{\bar{u}}{u}\bigr)^\alpha \bigr) \)
& 
& \( \displaystyle \frac{\mu}{\beta}  \frac{(1+s)^{\alpha-1}-s^{\alpha-1}}{s^\alpha-(1+s)^\alpha}  \)
\TBstrut\\
\hline
& \( \{\tau_\alpha\star \mathbb{1}_{[0,1]}\}(t) \)
& \( \displaystyle \rho_\alpha(x)\, \frac{1}{x}(1-e^{-x}) \)
& \( \displaystyle (1+s)^{\alpha}-s^{\alpha} \)
\TBstrut\\
\refstepcounter{rownum}{\scriptsize\therownum}\label{BFRYGGC} 
&   &     & $\big\downarrow$  $\big\uparrow$   \Bstrut\Bstrut\\
&  \( \displaystyle \frac{\alpha}{u\bar{u}} \, h_{1-\alpha,1-\alpha}\bigl(1 \vert  \bigl(\tfrac{\bar{u}}{u}\bigr)^\alpha \bigr) \)
& 
& \( \displaystyle \alpha  \frac{(1+s)^{\alpha-1}-s^{\alpha-1}}{s^\alpha-(1+s)^\alpha}  \)
\TBstrut\\
\hline
& 
& \( \displaystyle \frac{\mu}{x} \, e^{-x/2}\, I_\mu\left(\tfrac{x}{2}\right) \)
& \( \displaystyle \left(\sqrt{1+s}-\sqrt{s}\right)^{2\mu} \)
\TBstrut\\
\refstepcounter{rownum}{\scriptsize\therownum}\label{halfgenBFRYGGC} 
&   &   $\big\downarrow$   & $\big\downarrow$  $\big\uparrow$  \Bstrut\\
&  \( \displaystyle \mu \, {\rm beta}\bigl(u\vert \tfrac{1}{2},\tfrac{1}{2}\bigr) \)
& \( \mu \, e^{-x/2}\, I_0\left(\tfrac{x}{2}\right) \)
& \( \displaystyle \mu \frac{1}{\sqrt{s(1+s)}}  \)
\TBstrut\\
\hline

\end{tabular}
\end{center}
} 
\caption{Generic, stable and BFRY GGC commutative diagrams}
\label{table:GGCCD}
\end{table}

As in preceding tables, each column in Table~\ref{table:GGCCD} is the Laplace transform of the previous column
or the Stieltjes transform of the column before that.
The commutative diagram formed by the last two columns of 
Table~\ref{table:GGCCD}[case\ref{genGGC}] depict  an infinitely divisible (ID) distribution $f(x\vert \mu)$
(Steutel and van Harn~\cite[Chapter~III]{SteutelvanHarn}). 
If $\rho(x)$ is also the Laplace transform of a density $\tau(t)$,  known as  the Thorin density, 
then $f(x\vert \mu)$ is  a GGC.
Table~\ref{table:GGCCD}[case\ref{stableGGC}] is the stable GGC with $\mu=z$.
Table~\ref{table:GGCCD}[case\ref{genBFRYGGC}] is the  GGC arising from
Table~\ref{table:rvFreeDensities}[row\ref{free:mittagRatioPower[0,1]}] scaled by  $\mu>0$.
Table~\ref{table:GGCCD}[case\ref{BFRYGGC}] is the  case $\mu=1-\alpha$ studied by BFRY,
for which the GGC density is  $x^{-\alpha-1}(1-e^{-x})/\Gamma(1-\alpha)\equiv \rho_\alpha(x)(1-e^{-x})/x$. 
As noted by BFRY, this density (GGC setting aside) is a special case of Winkel~\cite{Winkel}.

By the convolution theorem, $\rho_\alpha(x)(1-e^{-x})/x$  is the  Laplace transform  of the convolution $\{\tau_\alpha\star \mathbb{1}_{[0,1]}\}(t)$ of the stable Thorin density 
$\tau_\alpha$ with the unit step-function of unit width $\mathbb{1}_{[0,1]}$ with  Laplace transform $(1-e^{-x})/x$.
The unit window sliding over the stable  Thorin density (as one might regard the convolution)
complements the BFRY  interpretation of the GGC distribution as representing 
 the excursion duration straddling an independent exponential time of a recurrent Bessel process.
 To make the Bessel connection  explicit, 
 we note that the modified Bessel function of order $1/2$ is 
 $I_{1/2}(x)=\sqrt{2/\pi x} \sinh(x) \implies
 1-e^{-x} = \sqrt{\pi x} \,e^{-x/2} I_{1/2}(x/2)$.
 Hence 
 \begin{align}
 \frac{x^{-\alpha-1}}{\Gamma(1-\alpha)} (1-e^{-x}) &= \rho_\alpha(x)\frac{1}{x}(1-e^{-x})
 =  \rho_\alpha(x)  \sqrt{\frac{\pi}{x}} \,e^{-x/2} I_{1/2}\bigl(\tfrac{x}{2}\bigr)
 \end{align}

Table~\ref{table:GGCCD}[case\ref{halfgenBFRYGGC}] $(\alpha=1/2)$ describes    random walks 
studied by Feller~\cite[Section~II.7]{Feller2}.
He showed that  the density of the first passage through integer $\mu>0$ 
 (the time  taken to reach height $\mu$ for the first time in a symmetric random walk in one dimension 
 starting at $0$) is  $\mu\,e^{-x}I_\mu(x)/x$ (where $x$ is time) 
and demonstrated its infinite divisibility, with $\rho(x)= e^{-x}I_0(x)$ (Feller~\cite[XIII.7]{Feller2}).
Without  reference to the Bessel function, 
Bondesson~\cite[Example~3.2.3]{Bondesson})  showed that the first passage distribution
is GGC with beta Thorin density ${\rm beta}\bigl(t\vert \tfrac{1}{2},\tfrac{1}{2}\bigr)$.
We also note the  limiting passage from the GGC density  $\sim I_\mu$ to its canonical density  $\sim I_0$  
without traversing the  Laplace transform route.

\section{Proofs of Table \ref{table:rvFreeDensities}}
\label{sec:proofs}
Table~\ref{table:rvFreeDensities}[rows\ref{free:stablePower}-\ref{free:stableRatioInv}]
have either been demonstrated or can readily be derived from preceding discussion.
 We present Table~\ref{table:rvFreeDensities}[row\ref{free:stableRatioPower3}]  next as a standalone proposition. 
 \begin{proposition}
\label{prop:stableRatioPowerDensity}
 Let $T = (S_{1-\alpha;z}/S_{1-\alpha})^{-\beta}$  $(0<\alpha<1, z>0)$, $\beta=(1-\alpha)/\alpha$. Then
\begin{align}
\Pr(T  \equiv t) 
&= \frac{1}{\beta t}  h_{1-\alpha,1-\alpha}(1 \vert z t^{\alpha})   
 \label{eq:stableRatioPowerDensity} \\
 \mathbb{E}\left[ \frac{1}{s+T} \right]
&=  \int_0^\infty \frac{1}{s+t} \, \Pr(t) \,dt
  \overset{\text{ac}}{=} \frac{1}{\beta} \frac{zs^{\alpha-1}}{1-zs^\alpha} 
 \label{eq:stableRatioPowerST} 
\end{align}
where $\overset{\text{ac}}{=}$ denotes equality under analytic continuation.
\end{proposition}
\begin{proof}[Proof of Proposition~$\ref{prop:stableRatioPowerDensity}$] 
\label{proof:stableRatioPowerDensity}
$\Pr(T)$ is simply a particular case of
Corollary~\ref{cor:stableMLproduct}:(\ref{eq:stableMLproductPowerdensity}) 
\begin{align*}
 \Pr((S_{1-\alpha;z}/S_{1-\alpha})^{-\beta} \equiv t)  
 &= \frac{1}{\beta t} h_{1-\alpha, 1-\alpha}(1 \vert zt^{(1-\alpha)/\beta}) \quad (\beta>0) \\ 
 &= \frac{1}{\beta t} h_{1-\alpha, 1-\alpha}(1 \vert zt^{\alpha}) \quad (\beta=(1-\alpha)/\alpha)
\end{align*}
The Stieltjes transform $\mathbb{E}[1/(s+T)]$ is significantly more challenging.
We adopt the same approach as in Proposition~\ref{prop:stablePowerDensity} to tackle the intractable 
integral~(\ref{eq:stableRatioPowerST}).
Accordingly, let $g_\alpha(t,u \vert z) \equiv u h_{1-\alpha, 1-\alpha}(u \vert zt^\alpha)/\beta t$ 
be a two-dimensional density on $(t>0,u>0)$,
with Laplace Transform $\phi_\alpha(x,u\vert z)$ over $t$.
Then $\Pr(t) = g_\alpha(t, u=1 \vert z)$ and 
\begin{align*}
\mathbb{E}\left[ \frac{1}{s+T} \right] = \int_0^\infty \frac{1}{s+t} \, g_\alpha(t,1 \vert z) \, dt 
  &= \int_0^\infty e^{-sx} \int_0^\infty  e^{-xt} \, g_\alpha(t,1 \vert z)  dt \, dx \\
  &= \int_0^\infty e^{-sx} \, \phi_\alpha(x,1\vert z) \, dx
 \end{align*}
We may then infer  $\phi_\alpha(x,u \vert z)$ from evaluating  the two dimensional Laplace transform of $g_\alpha(t,u \vert z)$
in two ways as in Proposition~\ref{prop:stablePowerDensity}.
We also recall that 
\begin{align*}
 h_{1-\alpha, 1-\alpha}(t \vert z) &=  \int_0^\infty  f_{1-\alpha}(t \vert z v) \, d{\rm ML}(v \vert 1-\alpha) \\
\mathbb{E}\left[M^k_{1-\alpha}\right]  &\equiv \int_0^\infty v^k d{\rm ML}(v \vert 1-\alpha) 
  = \frac{k!}{\Gamma(1+(1-\alpha)k)}
 \end{align*}
We first take the Laplace  transform over $u$ of $g_\alpha(t,u \vert z) \equiv u h_{1-\alpha, 1-\alpha}(u \vert zt^\alpha)/\beta t$ 
\begin{align*}
\int_0^\infty e^{-yu} g_\alpha(t,u \vert z)  du 
 &= \phantom{-} \frac{1}{\beta t} \int_0^\infty e^{-yu}\, u  \int_0^\infty  f_{1-\alpha}(u \vert zt^\alpha v) \, d{\rm ML}(v \vert 1-\alpha) du \\
 &= -\frac{1}{\beta t} \frac{d}{dy} \int_0^\infty  \int_0^\infty  e^{-yu}\, f_{1-\alpha}(u \vert zt^\alpha v) \, du \, d{\rm ML}(v \vert 1-\alpha)  \\
 &= -\frac{\alpha}{(1-\alpha) t} \frac{d}{dy} \int_0^\infty e^{-zt^\alpha y^{1-\alpha}v} \, d{\rm ML}(v \vert 1-\alpha) \\
 &= \phantom{-}  z\alpha t^{\alpha-1} y^{-\alpha} \sum_{k=0}^\infty \frac{(-zt^\alpha y^{1-\alpha})^k}{k!}  
                       \int_0^\infty v^{k+1} d{\rm ML}(v \vert 1-\alpha) \\
  &= -\frac{1}{y} \sum_{k=1}^\infty \frac{(-zy^{1-\alpha})^k }{\Gamma(1+(1-\alpha) k)}\, \alpha k \,  t^{\alpha k-1}
 \end{align*}

 Next,  take the Laplace  transform over $t$ 
\begin{align*}
  \psi_\alpha(x,y \vert z) &\equiv \int_0^\infty e^{-xt}  \int_0^\infty e^{-yu} \, g_\alpha(t,u \vert z)  du \, dt  
  \\
  &= -\frac{1}{y} \sum_{k=1}^\infty \frac{(-zy^{1-\alpha})^k }{\Gamma(1+(1-\alpha) k)} \, \alpha k \int_0^\infty e^{-xt}  t^{\alpha k-1} dt
  \\
  &= -\frac{1}{y} \sum_{k=1}^\infty \frac{(-zy^{1-\alpha})^k }{\Gamma(1+(1-\alpha) k)} \frac{\alpha k\Gamma(\alpha k)}{x^{\alpha k}} 
  \\
  &= -\frac{1}{\beta y} \sum_{k=1}^\infty  \left(\frac{-zy^{1-\alpha}}{x^{\alpha}}\right)^k  \frac{\Gamma(\alpha k)}{\Gamma((1-\alpha) k)}
\end{align*}
Now, we may equivalently swap the integration order, {\it i.e.}\ over $t$ first and then over $u$ 
\begin{align*}
 \psi_\alpha(x,y \vert z) &\equiv \int_0^\infty e^{-yu}  \int_0^\infty e^{-xt} \, g_\alpha(t,u \vert z)  dt \, du   \\
   &= \int_0^\infty e^{-yu} \,  \phi_\alpha(x,u \vert z) du
\end{align*}
We need to invert  this for $\phi_\alpha(x,u \vert z)$.
For values of \(k\) for which the integrals diverge at the origin, we interpret the following identities through the 
meromorphic continuation of the gamma-integral formula
\begin{align*}
y^{(1-\alpha)k-1} &\overset{\text{ac}}{=}  \frac{1}{\Gamma(1-(1-\alpha)k)} \int_0^\infty e^{-yu} \, u^{-(1-\alpha)k} \, du \\
\textrm{and}\quad
 s^{\alpha k-1} &\overset{\text{ac}}{=}  \frac{1}{\Gamma(1-\alpha k) } \int_0^\infty e^{-sx} \, x^{-\alpha k} \, dx \\
\implies \; \phi_\alpha(x,u \vert z)
 &\overset{\text{ac}}{=}  - \frac{1}{\beta} \sum_{k=1}^\infty  \left(\frac{-z}{u^{1-\alpha}x^{\alpha}}\right)^k  
 \frac{\Gamma(\alpha k)}{\Gamma((1-\alpha) k)\,\Gamma(1-(1-\alpha) k)}  \\
 \implies \; 
 \mathbb{E}\left[ \frac{1}{s+T} \right]  &=  \int_0^\infty e^{-sx} \phi_\alpha(x,1 \vert z) dx   \\
   &\overset{\text{ac}}{=}  - \frac{1}{\beta s} \sum_{k=1}^\infty  (-zs^{\alpha})^k 
\frac{\Gamma(\alpha k)\,\Gamma(1-\alpha k)}{ \Gamma((1-\alpha) k)\,\Gamma(1-(1-\alpha) k)}   \\
  &\overset{\text{ac}}{=}  - \frac{1}{\beta s} \sum_{k=1}^\infty  (-z s^{\alpha})^k  \, \frac{\sin\pi(1-\alpha)k}{\sin\pi\alpha k}  
    =   \frac{1}{\beta s} \sum_{k=1}^\infty  (z s^{\alpha})^k    \\
  &\overset{\text{ac}}{=}   \frac{1}{\beta} \, \frac{zs^{\alpha-1}}{1-zs^{\alpha}} 
     \qquad (\beta = (1-\alpha)/\alpha) 
\end{align*}
which proves~(\ref{eq:stableRatioPowerST}). 
We have used the Euler reflection formula $\Gamma(z)\Gamma(1-z)=\pi/\sin\pi z$ 
and the identity $\sin(\pi(1-\alpha)k)/\sin(\pi\alpha k) = (-1)^{k+1}$.
As discussed in Remark~\ref{rem:analyticContinuation}, the analytically  continued result~(\ref{eq:stableRatioPowerST}) 
is an intermediate step toward  Proposition~\ref{prop:freeBFRYDensity}.
\end{proof}


\begin{proposition}
\label{prop:freeLampertiDensity}
 Let $U = 1/(1+S_{\alpha;z}/S_\alpha)\in (0,1)$  where $(0<\alpha<1, z>0)$. 
 Then 
\begin{align}
\Pr(U  \equiv u) 
&= \int_0^\infty \delta\bigl(u-\tfrac{1}{1+t}\bigr)  h_{\alpha,\alpha}(t \vert z) dt 
= \frac{1}{u^2} h_{\alpha,\alpha}(\tfrac{1-u}{u} \vert z) 
\label{eq:freeLampertiDensity}  \\
 \mathbb{E}\left[ \frac{1}{s+U} \right]
&\equiv  \int_0^1 \frac{1}{s+u} \, \Pr(u) \, du   
= \frac{zs^{\alpha-1}+(1+s)^{\alpha-1}}{zs^\alpha+(1+s)^\alpha}  
\label{eq:freeLampertiDensityST} 
\end{align}
\end{proposition}

\begin{proof}[Proof of Proposition~$\ref{prop:freeLampertiDensity}$] 
\label{proof:freeLamperti}
By Corollary~\ref{cor:stableMLproduct}:(\ref{eq:stableMLproductDensity})
$\Pr(S_{\alpha;z}/S_\alpha \equiv T \equiv t) = h_{\alpha,\alpha}(t \vert z)$, hence
\begin{align*}
\Pr(U \equiv 1/(1+T) \equiv u)
&= \int_0^\infty \delta\bigl(u-\tfrac{1}{1+t}\bigr)  h_{\alpha,\alpha}(t \vert z) dt 
= \frac{1}{u^2} h_{\alpha,\alpha}(\tfrac{1-u}{u} \vert z) 
\end{align*} 
We recall from Table~\ref{table:rvFreeDensities}[row\ref{free:stableRatio2}] that
$\mathbb{E}[1/(s+T)] = s^{\alpha-1}/(z+s^\alpha)$ so that
\begin{align*}
 \mathbb{E}\left[ \frac{1}{s+U} \right]
&\equiv  \int_0^1 \frac{1}{s+u} \int_0^\infty \delta\bigl(u-\tfrac{1}{1+t}\bigr)  h_{\alpha,\alpha}(t \vert z) dt \, du   \\
&= \int_0^\infty \tfrac{1}{s+\tfrac{1}{1+t}}  \, h_{\alpha,\alpha}(t \vert z) dt   
  = \frac{1}{s} \int_0^\infty 1-\frac{1}{s} \frac{1}{(1+\tfrac{1}{s})+t}  \, h_{\alpha,\alpha}(t \vert z) dt   \\
&= \frac{1}{s} \left(1-\frac{1}{s} \frac{(1+\tfrac{1}{s})^{\alpha-1}}{z+(1+\tfrac{1}{s})^\alpha} \right) 
  = \frac{1}{s} \left(1-\frac{(1+s)^{\alpha-1}}{zs^\alpha+(1+s)^\alpha} \right) \\
&= \frac{zs^{\alpha-1}+(1+s)^{\alpha-1}}{zs^\alpha+(1+s)^\alpha} 
\end{align*} 
which proves the  Stieltjes transform~(\ref{eq:freeLampertiDensityST})
that Lamperti~\cite{Lamperti} used as a starting point.
\end{proof}

\begin{proposition}
\label{prop:freeBFRYDensity}
Let $U = 1/(1+(S_{1-\alpha;z}/S_{1-\alpha})^{-\beta})\in (0,1)$ $(\beta=(1-\alpha)/\alpha)$  where $(0<\alpha<1, z>0)$. 
 Then 
 \begin{align}
\Pr(U \equiv u) 
&= \frac{1}{\beta} \int_0^\infty \delta\bigl(u-\tfrac{1}{1+t}\bigr)  h_{1-\alpha,1-\alpha}(1 \vert z t^{\alpha}) \frac{dt}{t} \nonumber \\
&= \frac{1}{\beta} \frac{1}{u(1-u)} \, h_{1-\alpha,1-\alpha}\bigl(1 \vert  z \bigl(\tfrac{1-u}{u}\bigr)^\alpha\bigl)
\label{eq:freeBFRYDensity}  \\
 \mathbb{E}\left[ \frac{1}{s+U} \right]
&\equiv  \int_0^1 \frac{1}{s+u} \, \Pr(u) \, du   
= \frac{1}{\beta} \frac{z(1+s)^{\alpha-1}-s^{\alpha-1}}{s^\alpha-z(1+s)^\alpha}  
\label{eq:freeBFRYDensityST} 
\end{align}
\end{proposition}
\begin{proof}[Proof of Proposition~$\ref{prop:freeBFRYDensity}$]
\label{proof:freeBFRYDensity}
As   in  Proposition~\ref{prop:stableRatioPowerDensity}, 
 $\Pr((S_{1-\alpha;z}/S_{1-\alpha})^{-\beta}\equiv T)$ is given by 
Corollary~\ref{cor:stableMLproduct}:(\ref{eq:stableMLproductPowerdensity}) for $\alpha\to1-\alpha=\sigma; \beta=(1-\alpha)/\alpha$, hence
\begin{align*}
\Pr(U   \equiv 1/(1+T) \equiv u) 
&= \frac{1}{\beta} \int_0^\infty \delta\bigl(u-\tfrac{1}{1+t}\bigr)  h_{1-\alpha,1-\alpha}(1 \vert z t^{\alpha}) \frac{dt}{t} \nonumber \\
&= \frac{1}{\beta} \frac{1}{u(1-u)} \, h_{1-\alpha,1-\alpha}\bigl(1 \vert  z \bigl(\tfrac{1-u}{u}\bigr)^\alpha\bigl)
\end{align*}
We recall from Table~\ref{table:rvFreeDensities}[row\ref{free:stableRatioPower3}] that
$\mathbb{E}[1/(s+T)] \overset{\text{ac}}{=}  zs^{\alpha-1}/(1-s^\alpha)$, 
which is  the Laplace transform of $-E_\alpha(x^\alpha/z)$ so that 
$\int_0^\infty \Pr(t)dt  \overset{\text{ac}}{=}  -E_\alpha(0)=-1$.
Hence 
\begin{align*}
 \mathbb{E}\left[ \frac{1}{s+U} \right]
&\equiv  \frac{1}{\beta} \int_0^1 \frac{1}{s+u} 
    \int_0^\infty \delta\bigl(u-\tfrac{1}{1+t}\bigr) \, h_{1-\alpha,1-\alpha}(1 \vert z t^{\alpha}) \frac{dt}{t}\, du   \\
&= \frac{1}{\beta s} \int_0^\infty 1-\frac{1}{s} \frac{1}{(1+\tfrac{1}{s})+t}  \, h_{1-\alpha,1-\alpha}(1 \vert z t^{\alpha}) \frac{dt}{t}  \\
&= \frac{1}{\beta s} \left(-1-\frac{1}{s} \frac{z(1+\tfrac{1}{s})^{\alpha-1}}{1-z(1+\tfrac{1}{s})^\alpha} \right) 
  = \frac{1}{\beta s} \left(-1-\frac{z(1+s)^{\alpha-1}}{s^\alpha-z(1+s)^\alpha} \right) \\
&= \frac{1}{\beta} \frac{z(1+s)^{\alpha-1}-s^{\alpha-1}}{s^\alpha-z(1+s)^\alpha}  
\end{align*} 
which proves the  Stieltjes transform~(\ref{eq:freeBFRYDensityST})
\end{proof}

\section{Proofs of Table~\ref{table:rvSimpleDensities}}
\label{sec:simpleTable}

 The simple densities of Table~\ref{table:rvSimpleDensities} arise readily from substituting the  geometric series 
 representation~(\ref{eq:hGeometricSeries}) of $h_\alpha,\alpha(t \vert z)$  in the corresponding 
 representation-free expressions.

\section {Discussion and Conclusion}
\label{sec:discussion}

The core contribution of this paper has been the following:
\begin{enumerate}
\item Start with a ratio or power of stable random variables
\item Construct the induced representation-free density
\item Generate the forward  Stieltjes transform
\item Optionally construct,  bottom-up, the GGC structure arising therefrom
 \end{enumerate}
The exercise  is both representation-free and inversion-free.
None of the entries of Table~\ref{table:rvFreeDensities} involve Stieltjes inversion, 
the latter is restricted to the derivation of the well-known simple densities of 
 Table~\ref{table:rvSimpleDensities} using the infinite series representation of the stable density.
 
The Stieltjes transform for the representation-free density for the BFRY case  involved 
analytic continuation.
There may exist an elementary representation-free construction  of the Stieltjes transform that 
does not call for analytic continuation, thereby preserving a purely  probabilistic formulation from start to end.
Such a construction, if it exists, has  eluded us thus far.
Amongst other objectives, we intend to pursue this in future work.

 

\bibliography{./IntegralRepresentation}{}
\bibliographystyle{plain} 
\begin{acks}
The original idea of representation-free densities and their forward transforms  is purely ours.
We used  LLMs to check our pen and paper calculations. 
ChatGPT, in particular,  also volunteered  other, mostly stylistic, improvements.
We are yet to enlist LLM  assistance to search for  an  elementary approach 
that does not rely on  analytic continuation, thereby retaining an entirely probabilistic formulation.
\end{acks}
\end{document}